\documentclass[11pt]{article}
\usepackage{a4wide}
\usepackage[a4paper, margin=2.3cm]{geometry}
\usepackage{amsmath,amssymb,amsthm}
\usepackage{color}
\usepackage{parskip}
\usepackage{hyperref}
\usepackage{parskip}
\usepackage{setspace}
\usepackage{graphicx}
\usepackage{mathrsfs}

\usepackage{comment}

\newtheorem{theorem}{Theorem}[section]

\newtheorem{corollary}[theorem]{Corollary}

\newtheorem{lemma}[theorem]{Lemma}

\newtheorem{proposition}[theorem]{Proposition}
\newtheorem{definition}[theorem]{Definition}

\newtheorem*{definition*}{Definition}

\newcommand{\F}{\mathbb{F}}

\begin{document}

\title{A functional Loomis--Whitney type inequality in the Heisenberg group and projection theorems over finite fields}

\author{Daewoong Cheong\and 
Thang Pham
\and  Dung The Tran}

\date{}

\maketitle

\begin{abstract}
We establish functional Loomis--Whitney type inequalities in the finite Heisenberg group
$\mathbb{H}^n(\mathbb{F}_q)$. For $n=1$, we determine the sharp region of exponents $(u_1,u_2)$ for which the Heisenberg Loomis--Whitney inequality
\[
\frac{1}{q^3}\sum_{(x,t)\in \mathbb{H}^1(\mathbb{F}_q)}
f_1(\pi_1(x,t))\,f_2(\pi_2(x,t))
\;\lesssim\; \|f_1\|_{L^{u_1}(\mathbb{F}_q^2,dx)}\|f_2\|_{L^{u_2}(\mathbb{F}_q^2,dx)}
\]
holds uniformly in $q$, namely
\[
\frac{1}{u_1}+\frac{2}{u_2}\le 2
\quad\text{and}\quad
\frac{2}{u_1}+\frac{1}{u_2}\le 2,
\]
which includes the endpoint estimate $L^{\frac{3}{2}}\times L^{\frac{3}{2}}\to L^1$.
For general $n$, we prove the symmetric multilinear estimate at the endpoint exponent
$
u=\frac{n(2n+1)}{n+1},
$
using an induction on $n$ that exploits the Heisenberg fiber structure together with a multilinear interpolation scheme.
Specializing to indicator functions yields a sharp Loomis--Whitney type set inequality bounding $|K|$
for every finite $K\subset \mathbb{H}^n(\mathbb{F}_q)$ in terms of the sizes of its $2n$ Heisenberg projections
$\{\pi_j(K)\}_{j=1}^{2n}$, and in particular,
\[
\max_{1\le j\le 2n} |\pi_j(K)|
\;\gtrsim_n\;
|K|^{\frac{2n+1}{2(n+1)}}\,q^{-\frac{1}{2(n+1)}}.
\]
This result is optimal up to absolute constants. Moreover, when $n=1$ and $|K|>q$, we obtain a stronger statement via Vinh's point--line incidence theorem. We also discuss connections to a
boundedness problem for
multilinear forms/operators over finite fields studied by Bhowmik, Iosevich, Koh, and Pham (2025), and to orthogonal projection/covering questions in $\mathbb{F}_q^{2n+1}$ studied by Chen (2018).
\end{abstract}

\section{Introduction}\label{section-1}
The Loomis--Whitney inequality (\cite{LWR}) bounds the volume of a \(d\)--dimensional body in $\mathbb R^d$ by the volumes of its \((d-1)\)--dimensional coordinate projections. In the discrete setting, it reads as follows.  For  \(d\ge 2\), let \(B_1,\dots,B_d\) be finite sets, and let \(A\subseteq B_1\times\cdots\times B_d\). For $i\in \{1, \ldots, d\}$, define
\[
\pi_i\colon A\to B_1\times\cdots\times B_{i-1}\times B_{i+1}\times\cdots\times B_d,\qquad
\pi_i(a_1,\ldots,a_d):=(a_1,\ldots,a_{i-1},a_{i+1},\ldots,a_d).
\]
Then
\begin{equation}\label{eq:111}
 |A|^{\,d-1}\le \prod_{i=1}^d |\pi_i(A)| .
\end{equation}

There are powerful entropy-theoretic generalizations of 
\eqref{eq:111}, notably Han’s inequality and Shearer’s inequality \cite{Han, Ruz}. 
Variants of this inequality, in relation to set addition, also appear in additive combinatorics for abelian groups with numerous applications \cite{GMR, TV06}.

This paper investigates Loomis--Whitney type inequalities
in the noncommutative setting, focusing on the Heisenberg group over a finite field $\mathbb F_q$, where  $q$ is an odd prime power. In this context, such inequalities are not only of independent interest but also shed light on group structures and are connected to structural properties of the Heisenberg group and to related combinatorial problems. 
Related works on the Heisenberg group can be found in \cite{Pham, 3, HH, Shkredov}. Interestingly, this topic is naturally connected to boundedness problems for
multilinear forms/operators over finite fields 
and to orthogonal projection/covering questions in $\mathbb{F}_q^{2n+1}$. Further details will be discussed in Section~\ref{section7}.


In the continuous setting, the Loomis--Whitney type inequalities for the Heisenberg group (and, more generally, Carnot group) have been studied extensively with applications in multilinear restriction theory and related areas;  for example, see \cite{ bct, FP22, Zhang24} and references therein.

We now recall the definition of the Heisenberg group. For a positive integer $n$, the $n$-th Heisenberg group $\mathbb{H}^n(\mathbb{F}_q)$ is the set $\mathbb{F}^{2n+1}_q$ equipped with the group multiplication defined as follows:
for  $(x,t), (x', t') \in \mathbb{F}^{2n+1}_q$ with $x, x' \in \mathbb F_q^{2n}$ and $t, t'\in \mathbb F_q,$ 
\begin{align} \label{group opeartion}
    (x,t) \cdot (x^{\prime}, t^{\prime})=\left( x+x^{\prime}, t+t^{\prime}+ \frac{1}{2}\sum\limits_{j=1}^{n} ( x_j x^{\prime}_{n+j}-x_{n+j}x^{\prime}_j ) \right).
\end{align}
By definition, the group $\mathbb{H}^n(\mathbb{F}_q)$ and the vector space $\mathbb F_q^{2n+1}$ have the same underlying set. For each $1 \le j \le 2n$,  let 
\begin{equation} \label{coordinate-hyperplanes}
  W_j := \{(x,t)\in\F_q^{2n+1} : x_{j}=0\}, \qquad 
\end{equation}
be the coordinate hyperplane corresponding to $x_{j}$, and let
\begin{equation*}
  L_{j} := \{(s e_{j},0) \in \mathbb F_q^{2n+1} : s\in\F_q\}
\end{equation*}
be the coordinate line corresponding to $x_{j}$, where $\{e_j\}_{j=1}^{2n}$ form the standard basis of $\mathbb F_q^{2n}$. 

Then we easily see that $W_j$ and $L_j$ are subgroups of $\mathbb H^n(\mathbb F_q)$  for $j\in \{1,\ldots, 2n\}$. Further, $L_{j}$ 
is the orthogonal complement $W_j^\perp$ of $W_j$, with respect to the standard inner product on $\mathbb F_q^{2n+1}$; see Subsection \ref{subsection;complement} for these materials. 
The subgroups $W_j$ and $L_j$ will be  often referred to as vertical and horizontal, respectively.  We will provide more details on subgroup structures of $\mathbb H^n(\mathbb F_q)$ in Section \ref{section2-1}.

For $x \in \mathbb{F}^{2n}_q$ and $j \in \{1, \ldots, 2n\}$, we use the notation 
$\widehat{x}_j$ to represent the element in $\mathbb{F}^{2n-1}_q$ obtained by entirely removing the $j$-th coordinate of $x$ while the notation $\check{x}_j$ represents the element in $\mathbb{F}^{2n}_q$ obtained by setting the $j$-th coordinate of $x$ to $0$. For $j \in \{1, \ldots, n\}$, let
\begin{equation}
 \begin{cases}\label{sign}
      \pi_j: \mathbb H^n(\mathbb F_q) \rightarrow W_{j}\\
       \pi_{n+j}: \mathbb H^n(\mathbb F_q) \rightarrow W_{n+j}
    \end{cases}
\end{equation}
be the \textit{vertical coordinate projections}, respectively, defined by 
\begin{equation}\label{definition-pi-j}
 \begin{cases}
 \pi_j(x,t)=(\check{x}_{j}, t+\frac{1}{2} x_j x_{n+j}),\\
        \pi_{n+j}(x,t)=(\check{x}_{n+j}, t-\frac{1}{2} x_j x_{n+j}).
    \end{cases}
\end{equation}

A direct computation shows that for each $j\in \{1, \ldots, 2n\}$, an element $(x,t)$ of $\mathbb H^n(\mathbb F_q)$ can be written
\begin{equation}\label{eq:decomp}
 (x,t) = \pi_j(x,t)\cdot (x_{j} e_{j},0), \footnote{The signs $+$ and $-$ in (\ref{definition-pi-j}) are needed to guarantee this product holds uniformly for all $j$.}  
\end{equation}  
where the dot denotes the group multiplication on $\mathbb H^n(\mathbb F_q)$.

The fiber of $\pi_j$ at a point $(u, \tau)$ in $W_j$ is a line in the affine space $\mathbb F_q^{2n+1}$ passing through $(u, \tau)$.  More precisely,
for $j\in \{1, \ldots,n\}$,  the fiber of $\pi_j$ at  $(u, \tau) \in W_{j}$ (so $u_{j}=0$)  is given by
\begin{align*}
\pi_{j}^{-1}\{(u,\tau)\} 
&= \big\{\,\big(u + s e_{j},\; \tau - \tfrac12\,u_{n+j}\, s\big) :\; s\in\F_q\,\big\}
= (u,\tau)\cdot L_{j}, 
\end{align*}
and the fiber of $\pi_{n+j}$ at  $(u, \tau) \in W_{n+j}$ (so $u_{n+j}=0$) is given by
\begin{align*}
\pi_{n+j}^{-1}\{(u,\tau)\} 
&= \big\{\,\big(u + s e_{n+j},\; \tau + \tfrac12\,u_j\, s\big) :\; s\in\F_q\,\big\}
= (u,\tau)\cdot L_{n+j}.
\end{align*}
That is, for each $j\in \{1, \ldots, 2n\}$, the fibers of $\pi_j$  are exactly cosets of $W_j^\perp=L_{j}$. 

Given a set $K$ in $\mathbb{H}^n(\mathbb{F}_q)$, in this paper, we are interested in  the Loomis--Whitney type inequality associated with $K$, namely, bounding the size of $K$ in terms of the sizes of its projections, and its generalizations.

For any set $S$ (in most cases, $S=\mathbb F_q^{2n})$, denote by $\mathscr{F}(S)$ the set of all $\mathbb R$-valued functions on $S$. Note that $\mathscr{F}(S)$ is a vector space over $\mathbb R.$
For each $1 \le r \le \infty$, the $L^r$-norm on $\mathscr{F}(\mathbb F_q^{2n})$ 
is defined as
\[\|f\|_{L^r(\mathbb{F}_q^{2n}, dx)}=\left(\frac{1}{q^{2n}}\sum_{x\in \mathbb{F}_q^{2n}}|f(x)|^r\right)^{ \frac{1}{r} } \] 
for $1 \le r<\infty$, and for $r=\infty$ 
 \[\|f\|_{L^\infty(\mathbb{F}_q^{2n}, dx)}=\max\limits_{x \in \mathbb{F}_q^{2n}} |f(x)|.\]

\smallskip
Throughout, for each $j\in \{1,\ldots, 2n\}$, we shall identify $W_j$ with $\mathbb F_q^{2n}$ via the map $\phi_j: W_j\stackrel{\sim}{\rightarrow}\mathbb F_q^{2n}$ defined by $ (\check{x}_j, t)\mapsto (\widehat{x}_j,t)$. For simplicity, for a function $f$ on $\mathbb F_q^{2n}$ and $(x,t) \in \mathbb H^n(\mathbb F_q)$, we write $f(\pi_j(x,t))$ for $f(\phi_j (\pi_j (x,t)))$.

\begin{definition}\label{definition-LW-u-1-2}
Let $n$ be a positive integer. For extended real numbers $1 \le  u_1, u_2, \ldots, u_{2n} \le \infty$, define $LW(u_1, u_2, \ldots, u_{2n})$ as the smallest constant not depending on $q$, if it exists, such that the inequality
\begin{equation}\label{LW11}
\frac{1}{q^{2n+1}}\sum_{(x,t)\in \mathbb H^n(\mathbb F_q) } \prod_{j=1}^{2n} f_j(\pi_j(x,t))\le LW(u_1, u_2, \ldots, u_{2n}) \prod_{j=1}^{2n}  \|f_j \|_{L^{u_j}(\mathbb{F}_q^{2n}, dx)} 
\end{equation} 
holds for all nonnegative $f_1, f_2, \ldots, f_{2n} \in \mathscr{F}(\mathbb F_q^{2n}) $. 
If it does not exist, then we set  $LW(u_1, u_2, \ldots, u_{2n})=\infty.$
\end{definition}

Although our main results are stated as multilinear functional inequalities, they immediately imply
corresponding set/projection statements. When $f_1= \cdots= f_{2n}$ are all the indicator functions of a subset  $K$ of  $\mathbb H^n(\mathbb F_q)$, then  the inequality (\ref{LW11}) gives a Loomis--Whitney type inequality.

We first present a result in the case $n=1$.
\begin{theorem}\label{GenRL1}  For  $1\le u_1, u_2\le \infty,$ the following holds.  $$ LW(u_1, u_2)\lesssim 1 \quad \mbox{if and only if}\quad \frac{1}{u_1} +  \frac{2}{u_2}\le 2 \quad \mathrm{and} \quad  \frac{2}{u_1} +  \frac{1}{u_2}\le 2 .$$
\end{theorem}

In this paper, we write $X \lesssim Y$ to mean that there exists some constant $C>0$ such that $X \leq CY$.

\begin{corollary}\label{Theorem-L^3/2}
For any nonnegative \(f_1,f_2\in\mathscr{F}(\F_q^2)\),
\[
\frac{1}{q^3}\sum_{(x,t)\in \mathbb H^1(\F_q)}
f_1(\pi_1(x,t))\,f_2(\pi_2(x,t))
\;\lesssim\;
\|f_1\|_{L^{\frac{3}{2}}(\F_q^2,dx)}\,
\|f_2\|_{L^{\frac{3}{2}}(\F_q^2,dx)}.
\]
\end{corollary}

The endpoint exponent \(\frac{3}{2}\) is optimal for general functions \(f_1,f_2\) by Theorem~\ref{GenRL1}. 
On the level of sets, Corollary~\ref{Theorem-L^3/2} implies the bound
\begin{equation}\label{eq:setLW-n=1}
|K|\lesssim q^{\frac{1}{3}}\,|\pi_1(K)|^{\frac{2}{3}}\,|\pi_2(K)|^{\frac{2}{3}},
\qquad K\subset \mathbb H^1(\F_q),
\end{equation}
and this is sharp up to absolute constants. Indeed, fixing \(t_0\in\F_q\) and taking
\[
K:=\{(x_1,0,t_0)\in \mathbb{H}^1(\mathbb{F}_q)\colon x_1\in\F_q\},
\]
we have \(|K|=q\), \(|\pi_1(K)|=1\), and \(|\pi_2(K)|=q\), so \eqref{eq:setLW-n=1} is attained.

This construction can be extended to higher dimensions as follows. Let
\[
K:=\{(x,t_0)\in \mathbb H^n(\F_q): x_{n+1}=\cdots=x_{2n}=0\}.
\]
Then, \(|K|=q^n\), and
\[
|\pi_j(K)|=
\begin{cases}
q^{n-1}, & 1\le j\le n,\\[2pt]
q^{n}, & n+1\le j\le 2n.
\end{cases}
\]
This example naturally suggests the critical symmetric exponent in higher dimensions. Suppose one seeks a uniform estimate of the form
\begin{equation}\label{eq:symmetric-ansatz}
\frac{1}{q^{2n+1}}\sum_{(x,t)\in \mathbb H^n(\F_q)} \prod_{j=1}^{2n} f_j(\pi_j(x,t))
\;\lesssim\;
\prod_{j=1}^{2n}\|f_j\|_{L^u(\F_q^{2n},dx)}
\end{equation}
with a single exponent \(u\) for all inputs.
Applying \eqref{eq:symmetric-ansatz} to indicator functions \(f_j=\mathbf{1}_{\pi_j(K)}\) gives
\[
\frac{|K|}{q^{2n+1}}
\;\lesssim\;
\prod_{j=1}^{2n}\Bigg(\frac{|\pi_j(K)|}{q^{2n}}\Bigg)^{1/u}.
\]
Inserting the above values of \(|K|\) and \(|\pi_j(K)|\) forces
\[
q^n \;\lesssim\; q^{\,2n+1-\frac{2n^2+n}{u}},
\qquad\text{hence}\qquad
u \;\ge\; \frac{n(2n+1)}{n+1}.
\]
Thus, \(\frac{n(2n+1)}{n+1}\) is the smallest possible symmetric 
exponent one can expect when $n>1$.

When \(n>1\), the full description of the admissible region of \((u_1,\dots,u_{2n})\) for which \(LW(u_1,\dots,u_{2n})\lesssim 1\) holds is substantially more intricate than in the planar case \(n=1\). Rather than attempting a complete characterization, in this paper, we focus on the symmetric critical value \(u=\frac{n(2n+1)}{n+1}\), for which the inequality is established in the next theorem.

\begin{theorem}\label{high-dimension-2}   
    For any nonnegative $f_1, \ldots, f_{2n} \in \mathscr{F}(\mathbb{F}_q^{2n})$, we have
\begin{align}\label{main-inequality-2}
 \frac{1}{q^{2n+1}}\sum_{(x,t)\in \mathbb H^n(\mathbb F_q)} \prod_{j=1}^{2n} f_j(\pi_j(x,t)) 
 \lesssim
 \prod_{j=1}^{2n}  \|f_{j} \|_{L^{\frac{n(2n+1)}{n+1}}(\mathbb{F}_q^{2n}, dx)}.
\end{align}
\end{theorem}

The proof of the corresponding \(2n\)--multilinear inequality is not a formal iteration of the case \(n=1\). It is inspired by the induction (on dimension) scheme developed in the continuous setting \cite{FP22}, but requires nontrivial adaptations to the finite field setting and to obtaining constants uniform in $q$. It has two main components. First, we establish a family of asymmetric estimates (Theorem \ref{theorem-high-dimension}) in which one symplectic pair of inputs is measured in \(L^{\frac{2n+1}{2}}\) while the remaining inputs are measured in \(L^{2n+1}\). These bounds are proved by a delicate induction on \(n\) that repeatedly freezes one symplectic coordinate pair and rewrites the \(n\)--dimensional form as an averaged \((n-1)\)--dimensional Heisenberg Loomis--Whitney form. Second, we combine the resulting family of asymmetric bounds through a multilinear interpolation scheme to reach the symmetric estimate at the critical exponent \(u=\frac{n(2n+1)}{n+1}\).
    
As an application of Theorem \ref{high-dimension-2}, we obtain the following.

\begin{corollary}\label{Loomis-Whitney-inequality-F_p^q}
     For $K\subset \mathbb{H}^{n}(\mathbb{F}_q)$, we have 
\begin{align}\label{Loomis-Whitney-inequality}
 |K|\lesssim q^{\frac{1}{2n+1}} \prod_{j=1}^{2n} |\pi_j(K)|^{\frac{n+1}{n(2n+1)}}.
\end{align}
In particular, for 
$K\subset \mathbb{H}^1(\mathbb{F}_q)$,
\[\max\{|\pi_1(K)|, |\pi_2(K)|\}  \gtrsim |K|^{\frac{3}{4}}q^{-\frac{1}{4}}.\]
\end{corollary}
It follows from the example above that this corollary is sharp and we cannot replace the factor $q^{\frac{1}{2n+1}}$ by $|\pi_v(K)|^{\frac{1}{2n+1}}$, where  $\pi_v(K)$ is the image of   $K$ under the (vertical) projection $\pi_v: \mathbb H^n(\mathbb F_q) \rightarrow \mathbb F_q$ defined by $(x,t) \mapsto t $.

When $n=1$ and $|K|>q$, by using the point--line incidence theorem due to Vinh \cite{Vinh}, we are able to provide a better lower bound for the quantity $\max\{|\pi_1(K)|, |\pi_2(K)|\}$.
\begin{theorem}\label{Theorem-L-W-ineq:q-q^1/2}
    For $K\subset \mathbb{H}^{1}(\mathbb{F}_q)$, we have 
    \[|K|\lesssim \frac{|\pi_1(K)| \, |\pi_2(K)|}{q}+q^{ \frac{1}{2} }\sqrt{|\pi_1(K)| \, |\pi_2(K)|}.\]
In particular, 
\[  \max\{|\pi_1(K)|, |\pi_2(K)|\} \gtrsim \min \left\lbrace |K|^{ \frac{1}{2} }q^{ \frac{1}{2} }, ~|K| q^{- \frac{1}{2} }\right\rbrace.\]
\end{theorem}
This theorem is also sharp in its range. To see this, let $A,B\subset\F_q$ with $|A|=|B|=m$, and put $T=\F_q$.
Define
\[
K:=A\times B\times T \;\subset\; \mathbb H^1(\F_q).
\]
Then
\[
\pi_1(K)=B\times \F_q,\quad |\pi_1(K)|=mq,
\qquad
\pi_2(K)=A\times \F_q,\quad |\pi_2(K)|=mq,
\]
and
\[
|K|=|A|\,|B|\,|T|=m^2 q.
\]
Consequently,
\[
\frac{|\pi_1(K)|\,|\pi_2(K)|}{q}
\;+\; q^{\frac{1}{2}}\sqrt{|\pi_1(K)|\,|\pi_2(K)|}
\;=\; m^2 q \;+\; m q^{\frac{3}{2}},
\]
which matches $|K|$ up to absolute constants.

\section{Subgroups of the Heisenberg group and projections}\label{section2-1}
This section is devoted to the study of subgroups $G$ of the Heisenberg group $\mathbb{H}^n(\mathbb{F}_q)$ and the analysis of the projections onto relevant subgroups. The structural statements developed here will be used throughout the paper and provide tools for subsequent work.

\subsection{Basics on subgroups of $\mathbb H^n(\mathbb F_q)$}
When viewing a vector space $V$ over $\mathbb F_q$ as a group, we consider only its additive group structure,
forgetting scalar multiplication. 
Trivially, a subspace of the vector space $V$ is a subgroup of $V$. If the field $\mathbb F_q$ is a prime field, i.e., $q=p$ for some prime $p$, then the converse is also true.
Indeed, if $W$ is a subgroup of $V,$ then we can give a scalar multiplication on $W$ that makes $W$ into a subspace of $V.$ To see this, we write $\mathbb F_p=\{0,1,\ldots, p-1\}.$ For $t \in \mathbb F_p$ and $w \in W$, define $t w:=w+\cdots +w$ ($t$-times). We easily see that this coincides with the  scalar multiplication on $W$ induced from $V$, and so $W$ is a subspace of $V.$ If $\mathbb F_q$ is not a prime field, then there is a subgroup $W$ of $V$ that is not a subspace.  For example,  if $q=p^r$ for some $r\geq 2,$  the subfield $\mathbb F_p $ is a subgroup of the (one dimensional) vector space $\mathbb F_q, $ but not a subspace of $\mathbb F_q$ (over $\mathbb F_q$). Over a general $\mathbb F_q$ with $q=p^r,$  in order to obtain all subgroups of a vector space $V$, it is enough  to view $V$ as a vector space over $\mathbb F_p$ and find all subspaces of $V$ over $\mathbb F_p.$ Similarly,  linear maps between two vector spaces over a prime field are precisely group homomorphisms between the  corresponding underlying groups. 
For a subspace $W$  of the vector space $\mathbb F_q^{m},$ let $W^\perp$ be the subspace of $\mathbb F_q^{m}$  defined as 
$W^\perp=\{v \in \mathbb F_q^{m} : \langle v, w \rangle=0 \, \, \mathrm{for\, \, \,   all } \,  w\in W\}$, where 
$\langle \, ,\, \rangle$ is the standard inner product on $\mathbb F_q^{m}.$
We may call $W^\perp$ the orthogonal complement of $W$ in $\mathbb F_q^m$ with respect to $\langle \, ,\, \rangle$.

The quadratic term in the group multiplication \eqref{group opeartion} on  $\mathbb H^n(\mathbb F_q)$   plays an important role in understanding $\mathbb H^n(\mathbb F_q)$ and its subgroups. So we treat it especially.  For $x, x^\prime \in \mathbb F_q^{2n}$, define
\begin{equation*}\label{omega}
\omega( x, x^\prime):=\sum_{i=1}^n(x_i x_{n+i}^{\prime}-x_i^{\prime} x_{n+i}).
\end{equation*} Then  $\omega$ is a symplectic form, i.e., a nondegenerate skew-symmetric bilinear form on $\mathbb F_q^{2n}$. Note that $\omega$ is determined by
 the skew symmetric matrix $A_n$,  i.e., $\omega(x,x^\prime)=x \cdot A_n \cdot (x^\prime)^t,$ where
\[
A_n:=\begin{bmatrix}
   0&I_n \\
   -I_n&0 \\
    \end{bmatrix},
\]
and
$(\, )^t$ stands for the transpose of a matrix.

\begin{definition}\label{def1} A subspace $S$ of $\mathbb F_q^{2n}$ is called {\it isotropic} if  $\omega (x,x^\prime)=0$
for any two vectors $x, x^\prime \in S$.

\end{definition}
Note that  by linear algebra, maximal isotropic subspaces of $\mathbb F_q^{2n}$ have dimension $n$, and all subspaces of dimension one are isotropic.

Let $\pi_h:\mathbb H^n(\mathbb F_{q})\rightarrow \mathbb F_q^{2n}$ be the projection defined by $\pi_h(x, t)=x$.  Then $\pi_h$ is a group homomorphism and so if $G$ is a subgroup of 
$\mathbb H^n(\mathbb F_q),$ then  the image $\pi_h(G)$ is a subgroup of $\mathbb F_q^{2n}.$   Recall that  $\pi_v:\mathbb H^n(\mathbb F_{q})\rightarrow \mathbb F_q$ is the projection defined by $\pi_v(x, t)=t$. Note that $\pi_v$ is not a group homomorphism, but its restriction to the `vertical line' $\{(0,t)\in \mathbb H^n(\mathbb F_{q}) \, | \, t \in \mathbb F_q \}$ is a group isomorphism.

 Define an action of the group $\mathbb F_q^*$ on $\mathbb H^n(\mathbb F_q) $ by $s\cdot (x,t)=(s x, s^2 t)$ for $s \in \mathbb F_q^*$ and $(x,t) \in \mathbb H^n(\mathbb{F}_q).$ This action shall be called the dilation action, following \cite{Balog}.

\begin{definition}\label{def2}
A subgroup $G$ of $\mathbb H^n(\mathbb F_{q})$ is called {\it homogeneous} if $G$ is invariant under the  dilation action of  $\mathbb F_q^*$  on $\mathbb H^n(\mathbb F_{q})$.
\end{definition}
Recall that the orbit of an element $(x,t)\in \mathbb H^n(\mathbb F_{q})$ under the dilation action is given by
 \[\mathcal O_{\mathbb F_q^*}(x,t):=\{ s\cdot(x,t) \in \mathbb H^n (\mathbb F_q) \, | \, s \in \mathbb F_q^*\}.\] 
 If $G$ is any homogeneous subgroup of $\mathbb H^n(\mathbb F_{q})$ containing  
$(x, t)$,  then $G$ necessarily includes $\mathcal O_{\mathbb F_q^*}(x,t)$.

\begin{lemma}\label{lemma:orbit size}
For $(x,t)\in \mathbb H^n(\mathbb F_{q})$,  the cardinality of  $\mathcal O_{\mathbb F_q^*}(x,t)$ is given as follows.
\begin{itemize}
\item[(i)] If $x\ne 0$,  then $|\mathcal O_{\mathbb F_q^*}(x,t)|=q-1.$
\item[(ii)] If $x=0$ and $t\ne 0$, then  $|\mathcal O_{\mathbb F_q^*}(x,t)|=\frac{q-1}{2}.$
\item[(iii)] If $x=0$ and $t=0$,  then $|\mathcal O_{\mathbb F_q^*}(x,t)|=1.$
\end{itemize}
\end{lemma}

\begin{proof}
It can be easily verified that the isotropy groups of $(x,t)$ for the cases (i), (ii), (iii)  \[\mathrm{ISO}_{\mathbb F_q^*}(x,t)=\{s\in \mathbb F_q^* \, | \, s\cdot (x,t)=(x,t)\}\] are $\{1\},$ $\mathbb Z_2=\{-1,1\}$ and $\mathbb F_q^*$ , respectively. Then, the lemma is immediate from the formula 
 \[|\mathcal O_{\mathbb F_q^*}(x,t)| \cdot |\mathrm{ISO}_{\mathbb F_q^*}(x,t)|=|\mathbb F_q^*|=q-1.\]
\end{proof}

Now we will identify  and count all subgroups $G$ of $\mathbb H^n$ for two special cases, namely the case  when $\mathbb F_q$ is a prime field, and the case where subgroups  $G$ are homogeneous.

First assume that $\mathbb F_q$ is a prime field, say, $q=p$ for some prime $p.$ In this case we write $\mathbb F_p$ for $\mathbb F_q$.
First  fix a subgroup (equivalently,  a subspace) $S$ of $\mathbb F_p^{2n}.$ If $G$ is a subgroup of 
$\mathbb H^n(\mathbb F_p)$ such that $\pi_h(G)=S$,  by cardinality reason for finite groups, there are only two possibilities 
\[|G|=|S|,\]
\[|G|=|S|p.\]
If   $|G|=|S|p,$ then necessarily $G=S\times \mathbb F_p.$ 
If  $|G|=|S|,$ then $G$ is a graph  $\Gamma_S(\rho)$  of a map $\rho: S \rightarrow \mathbb F_p$,
 \[\Gamma_S(\rho):=\{(x,t)\subset \mathbb H^n(\mathbb{F}_p) \, |\, x\in S, t=\rho(x)\}.\]
But since $G=\Gamma_S(\rho)$ is a group, i.e., closed under the induced multiplication, $S$ should be isotropic and $\rho$ is a group homomorphism.
Summarizing, we have

\begin{proposition} \label{prop:subgroup} If $G$ is  a subgroup of $\mathbb H^n(\mathbb F_p)$ for a prime $p$, then $G$ is one of following two types
\begin{itemize}
\item[(a)]  $G=S\times \mathbb F_p$ for a  subspace $S$ of $\mathbb F_p^{2n}$. 
\item[(b)]   $G= \Gamma_S(\rho)$ for  a group  homomorphism (equivalently, a linear map) $\rho: S \rightarrow \mathbb F_p$, where $S$ is an isotropic  subspace of $\mathbb F_p^{2n}$. 
\end{itemize}
\end{proposition}

Next, let us  treat  homogeneous subgroups of $\mathbb H^n(\mathbb F_q)$. 
If $G$ is a homogeneous subgroup of $\mathbb H^n(\mathbb F_{q})$, then $\pi_h(G)$ is a subspace of $\mathbb F_q^{2n}$ (over $\mathbb F_q$). So, we fix a subspace  $S$ of  $\mathbb F_q^{2n}$, then find all subgroups $G$ of $\mathbb H^n(\mathbb F_q)$  such that $\pi_h(G)=S.$  Assume that $q=p^r$ for some prime $p$ and a positive integer $r$. The projection $\pi_h: \mathbb H^n(\mathbb F_{q})\rightarrow \mathbb F_q^{2n}$ induces a group homomorphism $\pi_h|_G:$ $G \rightarrow S.$ Then, elements of the kernel $K$ of $\pi_h|_G$ are of the form $(0,t)$ for $t \in \mathbb F_q$. Also, note that the restriction $\pi_v|_K:K \rightarrow \mathbb F_q$ embeds $K$ into $\mathbb F_q$ as a subgroup. Thus, $|K|$ divides $q=p^r$, and so $|K|=p^\ell$ for some $\ell \leq r.$  Firstly, assume that $K$ is nontrivial, i.e., contains an element $(0,t)$ with $t\ne0.$ Since $K$ is homogeneous, $K$ includes $\mathcal O_{\mathbb F_q^*}(0,t)$. So, by Lemma \ref{lemma:orbit size} (ii),
we  obtain $|K|\geq \frac{q-1}{2}= \frac{p^r-1}{2}$.  This forces $|K|=p^r=q$ since $p$ is odd. (Note that if $\ell\leq r-1$, then it would be that $|K| < \frac{q-1}{2}.$) Since $|G|=|S||K|=|S|q,$ in this case we have $G=S\times \mathbb F_q.$
Secondly, assume that $K=\{(0,0)\}$ is trivial.  Then $\pi_h|_G:G\rightarrow S$ is an isomorphism, and so, as in Proposition \ref{prop:subgroup}, $S$ is isotropic and $G= \Gamma_S(\rho)$ for some group homomorphism(equivalently, a linear map over $\mathbb F_p$)
$\rho:S\rightarrow \mathbb F_q.$ We claim  that $\rho$ is trivial.  To show this, let  $(x,\rho(x)) \in \Gamma_S(\rho)=G$ with $0\ne x\in \mathbb F_q^{2n}$. Then by the homogeneity condition,  for any $s\in \mathbb F_q^*$, we have  
\begin{equation} \label{equality-rho}
s\cdot(x, \rho(x))=(s x, s^2 \rho(x))=(sx, \rho(sx)).
\end{equation}   
But, in particular, if $s\in \mathbb F_p^*$, then $\rho(sx)=s\rho(x)$ since $\rho$ is linear over $\mathbb F_p.$ By this fact and \eqref{equality-rho}, we obtain the equality
$s^2 \rho(x)=s \rho(x)$ for any  $s\in \mathbb F_p^*$. This implies that $\rho(x)=0$ for all nonzero $x\in \mathbb F_q^{2n}$, i.e., $\rho$ is trivial. Therefore, $G=\Gamma_S(\rho)=S\times \{0\}.$  In conclusion,  we have proved
\begin{proposition} \label{prop: homogenous subgroup} If $G$ is a  homogeneous subgroup of $\mathbb H^n(\mathbb F_q)$,  then $G$ is one of the following two cases 
\begin{itemize}
\item[(a)] $G=S\times \mathbb F_q$ for a subspace $S$ of $\mathbb F_q^{2n}$. 
\item[(b)]    $G=S\times \{0\}$ for an isotropic  subspace $S$ of $\mathbb F_q^{2n}$. 
\end{itemize}
\end{proposition}

\subsection{Orthogonal complement of a homogeneous subgroup of $\mathbb H^n(\mathbb F_q)$}\label{subsection;complement}
By Proposition \ref{prop: homogenous subgroup}, homogeneous subgroups $G$ of $\mathbb H^n(\mathbb F_q)$ are some subspaces of $\mathbb F_q^{2n+1}$ via the identification $\mathbb H^n(\mathbb F_q)=\mathbb F_q^{2n+1}$. By abuse of notation, through this identification,  we take $G^\perp$  for homogeneous subgroups $G$ of $\mathbb H^n(\mathbb F_q)$. If $G=S\times \{0\}$ for an isotropic  $S$  is a homogeneous subgroup,  then
we have $G^\perp=S^\perp \times \mathbb F_q $, which is necessarily a (vertical) homogeneous subgroup of  $\mathbb H^n(\mathbb F_q)$.
If $G=S\times \mathbb F_q$, where $S$ is an arbitrary subspace of $\mathbb F_q^{2n}$, then $G^\perp=S^\perp \times \{0\}$. In this case $G^\perp$  is not necessarily a homogeneous subgroup of $\mathbb H^n(\mathbb F_q)$. For example, if $\mathrm{dim}\, S \le n-1$, then $\mathrm{dim} \, S^\perp \ge n+1 $ and so $S$ is not isotropic by the dimension reason. Thus $G^\perp=S^\perp \times \{0\}$ is not a homogeneous subgroup of $\mathbb H^n(\mathbb F_q)$. (In fact, it is not even a subgroup of $\mathbb H^n(\mathbb F_q)$.)

\subsection{Counting subgroups of  $\mathbb H^n(\mathbb F_q)$ }
Now we count subgroups of $\mathbb H^n(\mathbb F_q)$ for the two special cases in the above.
For $k=0, \ldots, 2n$, let $\mathrm{Gr}(k, 2n)$ be the variety of $k$ dimensional subspaces in $\mathbb F_q^{2n}.$ For $k=0,\ldots ,n$, let
$\mathrm{IG}(k, 2n)$ be the variety of $k$ dimensional isotropic subspaces in $\mathbb F_q^{2n}$.

First we count subgroups of $\mathbb H^n(\mathbb F_p)$ for a prime $p$, using Proposition \ref{prop:subgroup}.
For a subspace $S$  of $\mathbb F_p^{2n}$ of dimension $k$ ($|S|=p^k$), let $n(S)$(resp. $n(S)^\prime$) be the number of subgroups $G$ of $\mathbb H^n(\mathbb F_p)$ such that $\pi_h(G)=S$ and $|G|=p^k$ (resp. $|G|=p^{k+1}$). Then, by Proposition \ref{prop:subgroup}, the number of all subgroups of $\mathbb H^n(\mathbb F_p)$ equals
\[  \sum_{k=0}^{2n}\sum _{S\in \mathrm{Gr}(k, 2n)} |n(S)^\prime| +  \sum_{k=0}^{n}\sum _{S\in \mathrm{IG}(k,  2n)} |n(S)|. \]

We easily see that $n(S)^\prime=1$ for each $S \in \mathrm{Gr}(k, 2n) $; and $n(S)=p^k$ for each $S \in \mathrm{IG}(k, 2n) $ because there are $p^k$ group homomorphisms (linear maps) from $S$ to $\mathbb F_p$, identified with $1\times k$ matrices with entries in $\mathbb F_p$.

This number equals
\begin{equation} \label{count1} \sum_{k=0}^{2n}|\mathrm{Gr}(k, 2n)| +  \sum_{k=0}^{n} p^k | \mathrm{IG}(k,  2n)|. 
\end{equation}

 Next, we count homogeneous subgroups of $\mathbb H^n(\mathbb F_q)$, using Proposition \ref{prop: homogenous subgroup}. 
Fix a subspace $S$ of $\mathbb F_q^{2n}$ of dimension $k$ (so $|S| = q^k$). For such $S$, let $m(S)$ (resp.\ $m(S)'$) be the number of homogeneous subgroups $G$ with $\pi_h(G) = S$ and $|G| = q^k$ (resp.\ $|G| = q^{k+1}$). 
Then, by Proposition \ref{prop: homogenous subgroup}, the number of all homogeneous subgroups of $\mathbb H^n(\mathbb F_q)$  equals

\[  
\sum_{k=0}^{2n}\sum _{S\in \mathrm{Gr}(k, 2n)} |m(S)^\prime| +  \sum_{k=0}^{n}\sum _{S\in \mathrm{IG}(k,  2n)} |m(S)|. 
\]
 Since $m(S)^\prime=1$ for each $S \in \mathrm{Gr}(k, 2n)$; and $m(S)=1$ for each  $S \in \mathrm{IG}(k, 2n) $, the  
 number  of all homogeneous subgroups of $\mathbb H^n(\mathbb F_q)$ equals
\begin{equation}\label{count2}
\sum_{k=0}^{2n}|\mathrm{Gr}(k, 2n)| +  \sum_{k=0}^{n}  | \mathrm{IG}(k,  2n)|. 
\end{equation}

Now, by \eqref{count1} and \eqref{count2}, enumerating subgroups of $\mathbb H^n(\mathbb F_q)$ is reduced to counting elements of   $|\mathrm{Gr}(k, 2n)|$ and  $|\mathrm{IG}(k, 2n)|$. Let us give counting formulas for elements of these varieties.  
For an indeterminate $q$ and an integer $m>0$, let $[m]_q$ be a polynomial in $q$ 
\[[m]_q:=\sum_{i=0}^{m-1} q^{i}=\frac{1-q^m}{1-q},\]
and set \[[m]_q!:=[m]_q [m-1]_q \cdots [2]_q[1]_q.\]
Then, the following formulas can be found in the literature, for example, in \cite{Carter, Wan}.

\begin{proposition}
The cardinalities of   $\mathrm{Gr}(k, 2n)$ and  $\mathrm{IG}(k, 2n)$, respectively,  are given by  
 \[|\mathrm{Gr}(k, 2n)|=\frac{[2n]_q!}{[k]_q![(2n-k)]_q!},\]
\[|\mathrm{IG}(k, 2n)|=\frac{[n]_q!}{[k]_q![n-k]_q!}   \cdot \prod _{i=n-k+1}^n\left ( q^i+1\right).\]
\end{proposition}

In Proposition \ref{prop: homogenous subgroup}, we dealt with only homogeneous subgroups of $\mathbb H^n(\mathbb F_q)$.  Using the idea in Propositions \ref{prop:subgroup} and \ref {prop: homogenous subgroup}  we can still construct some subgroups of $\mathbb H^n(\mathbb F_q)$. For example, fix a subgroup $S$ of $\mathbb F_q^{2n}$ so that $S$ is a subspace over $\mathbb F_p$. Let $F(S)$ be the smallest subfield of $\mathbb F_q$ containing all components $x_i \in \mathbb F_q $ of $x=(x_1, \ldots, x_{2n})$, where $x$ runs through $S$. Then for any subfield $F$ of $\mathbb F_q$ with $F(S)\subset F$, the product $S\times F$ is a subgroup of $\mathbb H^n(\mathbb F_q)$. Also, if $S$ is isotropic, then as in Proposition \ref{prop:subgroup}, we can obtain a subgroup $G$ of $\mathbb H^n(\mathbb F_q)$ as a graph of a group homomorphism. The group thus constructed is of the non-product type. In case $S$ is not isotropic, we could not find a  subgroup $G$  of  the non-product type, with $\pi_h(G)=S$, and we do not know whether there exists such a subgroup $G$.  It would be interesting to investigate existence of such subgroups of the non-product type, with $\pi_h(G)=S$ for  a non-isotropic subgroup $S$, or, more generally, to classify and count  all subgroups of $\mathbb H^n(\mathbb F_q)$.

In conclusion, the coordinate subgroups $W_1, \dots, W_{2n}$ discussed here form a homogeneous family that will be central to our main results. A detailed analysis of other homogeneous subgroups will be reserved for future work, to maintain focus and clarity in the current discussion.

\section{Proof of Theorem \ref{GenRL1}}\label{section2}
 Define a bilinear operator (form) $\mathcal{L}$ on  $\mathscr{F}(\mathbb F_q^2)$  by
 \begin{equation}\label{lineDe} \mathcal{L}(f_1,f_2):= \frac{1}{q^3} \sum_{x=(x_1,x_2), y=(y_1,y_2)\in \mathbb F_q^2} 1_{x_1\cdot y_1+ y_2=x_2} f_1(x) f_2(y),\end{equation}
 for  any $f_1, f_2 \in \mathscr{F}(\mathbb F_q^2).$ Let $\langle \, ,  \, \rangle$ be the `standard' inner product on $\mathscr{F}(\mathbb F_q^2)$, i.e., 
 \[\langle f_1, f_2\rangle=\frac{1}{q^2}\sum_{x=(x_1, x_2)\in \mathbb F_q^2}f_1(x)f_2(x).\]
 Then we can write $\mathcal{L}(f_1, f_2)=\langle f_1, A f_2\rangle$, where $A$ is an endomorphism on $\mathscr{F}(\mathbb F_q^2)$ defined by \[\left(A f_2\right)(x):=  \frac{1}{q} \sum_{y=(y_1,y_2) \in\mathbb F_q^2} 1_{x_1y_1 +y_2= x_2} f_2(y) \] for $f \in \mathscr{F}(\mathbb F_q^2)$. 
Note that the bilinear operator $\mathcal{L}$ is not symmetric.

For $1 \le r, s \le \infty,$  denote  by $A(s \to r)$  the smallest number, if it exists, such that the inequality 
\begin{align*}
    \|Af\|_{L^r(\mathbb{F}^2_q, dx)}
    \leq A(s \to r) \| f \|_{L^s(\mathbb{F}^2_q, dx)} 
\end{align*} 
holds for all nonnegative $f \in \mathscr{F}(\mathbb F_q^2).$ 

To prove Theorem \ref{GenRL1}, 
we introduce a constant $L(u_1, u_2)$ as an alternative to the constant $LW(u_1, u_2).$
\begin{definition}\label{definition-L-u-u}
For $1\le u_1, u_2\le \infty,$ 
define $L(u_1, u_2)$ as the smallest positive constant, if it exists, such that the following estimate holds for all nonnegative $f_1, f_2 \in \mathscr{F}(\mathbb F_q^2)$
\begin{equation} \label{inequality-L} \mathcal{L}(f_1,f_2) \le L(u_1, u_2) \|f_1\|_{L^{u_1}(\mathbb{F}^2_q, dx)} \cdot \| f_2\|_{L^{u_2}(\mathbb{F}^2_q, dx)},
\end{equation}
and if it does not exist, we set $L(u_1, u_2)=\infty.$
\end{definition} 

Recall that for $1\leq u \le \infty$, the H\"{o}lder's conjugate of $u$ is defined by 
$\frac{1}{u}+\frac{1}{u^\prime}=1$.

\begin{lemma} \label{lemma:alternative constant}
If  $1\le u_1, u_2 \le \infty$ satisfy  $A(u_2 \to u_1')\lesssim 1$, then  we have $L(u_1, u_2) \lesssim 1.$
\end{lemma}

\begin{proof}
By H\"older's inequality and the assumption, we have 
$$ \mathcal{L}(f_1, f_2)=\langle f_1, Af_2 \rangle \le \|f_1\|_{L^{u_1}(\mathbb{F}^2_q, dx)} \,  \|Af_2\|_{L^{u'_1}(\mathbb{F}^2_q, dx)} \lesssim \|f_1\|_{L^{u_1}(\mathbb{F}^2_q, dx)} \,  \|f_2\|_{L^{u_2}(\mathbb{F}^2_q, dx)}. $$
Since these inequalities hold for any $f_1, f_2 \in \mathscr{F}(\mathbb F_q^2)$, the lemma is immediate.
\end{proof}

\begin{theorem}\label{GenRL} 
For $1\le u_1, u_2\le \infty,$  the following holds.
$$ L(u_1, u_2)  \lesssim  1 \quad \mbox{if and only if}\quad \frac{1}{u_1} +  \frac{2}{u_2}\le 2\quad \mathrm{and} \quad  \frac{2}{u_1} +  \frac{1}{u_2}\le 2. $$
\end{theorem}
\begin{proof}
First, assume $L(u_1, u_2) \lesssim 1$.  To obtain  the first inequality $\tfrac{1}{u_1} + \tfrac{2}{u_2} \leq 2$, take two test functions
\[
f_1(x_1, x_2) = \mathbf{1}_{\{x_2=x_1+1\}}, 
\quad 
f_2(x_1, x_2) = \mathbf{1}_{\{x_1=x_2=1\}}.
\]
By simple calculations, we obtain
\[
\mathcal L(f_1, f_2) = q^{-2}; 
\]
\[
\|f_1\|_{L^{u_1}(\mathbb{F}^2_q, dx)} = \Bigl(\frac{q}{q^2}\Bigr)^{1/u_1} = q^{-1/u_1}, 
\qquad 
\|f_2\|_{L^{u_2}(\mathbb{F}^2_q, dx)} = q^{-2/u_2}.
\]
The assumption $L(u_1, u_2) \lesssim 1$ gives
\[
q^{-2} \lesssim q^{-1/u_1 - 2/u_2},
\]
which is equivalent to the inequality $\tfrac{1}{u_1} + \tfrac{2}{u_2} \leq 2$.

The second inequality $\tfrac{2}{u_1} + \tfrac{1}{u_2} \leq 2$ can be obtained in a similar manner by choosing 
\[
g_1(x_1, x_2) = \mathbf{1}_{\{x_1=x_2=1\}}, 
\quad 
g_2(x_1, x_2) = \mathbf{1}_{\{x_2=1-x_1\}}.
\]
For the other direction,  assume that  $1\le u_1, u_2\le \infty$ satisfy the inequalities
$$\frac{1}{u_1} +  \frac{2}{u_2}\le 2 \quad \mathrm{and}  \quad  \frac{2}{u_1} +  \frac{1}{u_2}\le 2.$$ Let us show that $L(u_1, u_2)\lesssim 1$. 
By the interpolation theorem and the nesting property of  the norm, it suffices to establish the estimates on  the critical end-points $(\frac{1}{u_1}, \frac{1}{u_2}) \in [0, 1]\times [0,1]$, which are     
$ (0,1), (1,0),$ and $(\frac{2}{3}, \frac{2}{3})$. In other words,  it remains to prove the following estimates:
\begin{align*}
     L(\infty, 1)   \lesssim   1,   ~L(1, \infty)   \lesssim   1, ~ L\left(\frac{3}{2}, ~\frac{3}{2}\right)   \lesssim   1. 
\end{align*}
But these estimates follow directly from Lemma \ref{lemma:alternative constant}  and the following result,  due to Koh \cite[Theorem 1.1]{koh}:
$$ A(1\to 1)   \lesssim 1,  ~~A(\infty\to \infty)  \lesssim 1,  ~~ A\left(\frac{3}{2} \to 3\right)  \lesssim 1.$$
\end{proof}

The following result is a special case of Theorem \ref{GenRL}, but it is very useful in practice.
\begin{corollary}\label{LamRes} 
We have
$ L\left( \frac{3}{2},  \frac{3}{2}\right)  \lesssim 1.$
\end{corollary}

\begin{proof}
The proof follows immediately from Theorem \ref{GenRL} since 
$$\frac{1}{u_1}+\frac{2}{u_2}=\frac{2}{u_1} +\frac{1}{u_2}= 2 $$
with $u_1=u_2=\frac{3}{2}$.
\end{proof}

Now we can compare two constants $LW(u_1,u_2)$ and $L(u_1, u_2).$
\begin{lemma}\label{remark:L-LW-u-u} For  $1\le u_1, u_2 \le \infty$, the following holds.
     \[L(u_1, u_2)\lesssim 1 \quad \mathrm{if \, \,   and \,  \, only\, \, if}\quad LW(u_1, u_2)\lesssim 1.\]
\end{lemma} 
\begin{proof} To prove the lemma, it is sufficient to show the LHS of the inequality \eqref{LW11}  with $n=1$ is equal to LHS, $\mathcal L(f_1,f_2)$, of the inequality \eqref{inequality-L}. 
 In fact, it follows from \eqref{connection1} that 
\begin{align*}
 &\frac{1}{q^3}\sum_{(x_1, x_2, t)\in \mathbb{H}_q^1}
f_1(\pi_1(x_1, x_2, t))f_2(\pi_2(x_1, x_2, t))
    =\frac{1}{q^3}\sum_{x_1, x_2, t\in \mathbb{F}_q}f_1(x_2, t+x_1x_2)f_2(x_1, t)
    =\mathcal{L}(f_1, f_2),
\end{align*}
as desired.
\end{proof}
In conclusion, Theorem \ref{GenRL1}  follows from Theorem \ref{GenRL} and Lemma \ref{remark:L-LW-u-u}



\section{Proof of Theorem \ref{high-dimension-2}}
While proving Theorem \ref{high-dimension-2}, we will heavily use the following result. 
\begin{theorem}\label{theorem-high-dimension}
Fix any $k=1, \ldots, n$. Then for any nonnegative $f_1, \ldots, f_{2n} \in \mathscr{F}(\mathbb F_q^{2n}),$ the inequality
\begin{align}\label{main-inequality}
\nonumber
   & \frac{1}{q^{2n+1}}\sum_{(x,t)\in \mathbb H^n(\mathbb F_q)} \prod_{j=1}^{2n} 
   f_j(\pi_j(x,t)) \\ 
& \lesssim  \|f_{k} \|_{L^{\frac{2n+1}{2}}(\mathbb{F}_q^{2n}, dx)}  \|f_{n+k} \|_{L^{\frac{2n+1}{2}}(\mathbb{F}_q^{2n}, dx)} \prod_{j=1, j \ne k}^{n}\left(  \|f_{j} \|_{L^{2n+1}(\mathbb{F}_q^{2n}, dx)}  \|f_{n+j} \|_{L^{2n+1}(\mathbb{F}_q^{2n}, dx)} \right)
\end{align} holds.  
\end{theorem}

\begin{proof}

We may assume, WLOG, that $k=1$. 
Denote
\[
I:=\frac{1}{q^{2n+1}}\sum_{(x,t)\in \mathbb H^n(\mathbb F_q)}
\prod_{j=1}^{2n} f_j(\pi_j(x,t)).
\]

{\bf Step $1^{\circ}$ (Separate the factor $f_{n}$.)} 
For convenience, for $x=(x_1, \ldots, x_{2n}) \in \mathbb F_q^{2n}$, write $x=(x',y)$, where $x'=(x_1, \ldots, x_{2n-1})\in\mathbb{F}_q^{2n-1}$ and $y=x_{2n} \in \mathbb F_q$, and the change of variables $\tau:=t-\frac{1}{2}x_n y$.  Then we can write
\begin{align}\label{definition-I} 
\nonumber
I
&=\frac{1}{q^{2n+1}}\sum_{(x',y,t) \in \mathbb{H}_q^{n}}
f_{2n}(x',t -\frac{1}{2}x_n y)
f_{n}(\widehat{x}_{n},t+\frac{1}{2}x_n y)\,  
\prod_{\substack{j=1\\ j\neq n,2n}}^{2n}
f_j\big(\pi_j(x',y,t)\big)
\\ \nonumber
&=\frac{1}{q^{2n+1}} \sum_{(x', y, \tau) \in \mathbb{F}_q^{2n+1}}
f_{2n}(x', \tau)\,  f_{n}(\widehat{x}_{n}, \tau+x_n y) 
\prod_{\substack{j=1\\ j\neq n,2n}}^{2n}
f_j\big(\pi_j(x',y,\tau+\frac{1}{2}x_n y)\big)\\
&=\frac{1}{q^{2n}} \sum_{(x',\tau) \in \mathbb{F}_q^{2n}}
f_{2n}(x', \tau)\, \frac{1}{q} \sum_{y \in \mathbb{F}_q}  f_{n}(\widehat{x}_{n}, \tau+x_n y) 
\prod_{\substack{j=1\\ j\neq n,2n}}^{2n}
f_j\big(\pi_j(x',y,\tau+\frac{1}{2}x_n y)\big).
\end{align} 
Note that $f_{2n}$ is independent of $y$. By applying H\"{o}lder's inequality with the conjugate pair of exponents $(2n+1, \frac{2n+1}{2n})$, 
we obtain 
\begin{align*}
I \leq &\bigg( \frac{1}{q^{2n}} \sum_{(x', \tau) \in \mathbb{F}_q^{2n}}
f_{2n}(x', \tau)^{2n+1}  \bigg)^{\frac{1}{2n+1}} \cdot J 
= \| f_{2n} \|_{L^{2n+1}(\mathbb{F}_q^{2n}, dx)} \cdot J,
\end{align*}
where $J$ captures the contribution of the remaining product 
\[
J:=\left[ \frac{1}{q^{2n}}\sum_{(x',\tau) \in \mathbb{F}_q^{2n}} \bigg( \frac{1}{q}\sum_{y \in \mathbb{F}_q}
 f_{n}(\widehat{x}_n, \tau+x_n y) 
\prod_{\substack{j=1\\ j\neq n,2n}}^{2n}
f_j\big(\pi_j(x',y,\tau+\frac{1}{2}x_n y) \big)  \bigg)^{\frac{2n+1}{2n}} \right]^{\frac{2n}{2n+1}}.
\]
 To prove \eqref{main-inequality}, it is enough to show that
\begin{align*}
 J \lesssim  \| f_{1} \| _{L^{\frac{2n+1}{2}}(\mathbb{F}_q^{2n}, dx)}  \| f_{n+1} \| _{L^{\frac{2n+1}{2}}(\mathbb{F}_q^{2n}, dx)} 
  \| f_{n} \| _{L^{2n+1}(\mathbb{F}_q^{2n}, dx)} \prod_{j=2}^{n-1}\left(  \| f_{j} \| _{L^{2n+1}(\mathbb{F}_q^{2n}, dx)}  \| f_{n+j} \| _{L^{2n+1}(\mathbb{F}_q^{2n}, dx)} \right).
\end{align*}

{\bf Step $2^{\circ}$ (Extract the $f_n$-factor).}
For fixed $y$ and $x_n$, perform the change of variables $s:=\tau+x_n y$.
For $x = (x_1, \ldots, x_{2n}) \in \mathbb{F}_q^{2n}$, we use the notation $\widehat{x}_{i,j}$ to denote the element in $\mathbb{F}_q^{2n-2}$ obtained by removing $x_i$ and $x_j$ from $x$. 
By Minkowski inequality, we obtain 
\begin{align*}
    J=& \left[ \frac{1}{q^{2n-1}}\sum_{x' \in \mathbb{F}_q^{2n-1}} \frac{1}{q} \sum_{s \in \mathbb{F}_q} \bigg( \frac{1}{q}\sum_{y \in \mathbb{F}_q}
 f_{n}(\widehat{x}_n, s) 
\prod_{\substack{j=1\\ j\neq n,2n}}^{2n}
f_j\big(\pi_j(x',y, s-\frac{1}{2}x_n y) \big) \bigg)^{\frac{2n+1}{2n}}  \right]^{\frac{2n}{2n+1}} \\
    \leq & \frac{1}{q}\sum_{y \in \mathbb{F}_q} \bigg[ \frac{1}{q^{2n-1}}\sum_{x' \in \mathbb{F}^{2n-1}_q }   \frac{1}{q}\sum_{s\in \mathbb{F}_q} f_n(\widehat{x}_n,s)^{\frac{2n+1}{2n}}\prod_{\substack{j\neq n,2n}} f_j\big(\pi_j(x',y, s -\frac{1}{2} x_n y)\big)^{\frac{2n+1}{2n}} \bigg]^{\frac{2n}{2n+1}} \\
    =& \frac{1}{q}\sum_{y \in \mathbb{F}_q} \bigg[ \frac{1}{q^{2n-1}}\sum_{(\widehat{x}_{n, 2n},s) \in \mathbb{F}^{2n-1}_q }  f_n(\widehat{x}_n,s)^{\frac{2n+1}{2n}}  \bigg(  \frac{1}{q}\sum_{x_n\in \mathbb{F}_q}\prod_{\substack{j\neq n,2n}} f_j\big(\pi_j(x',y, s -\frac{1}{2} x_n y)\big)^{\frac{2n+1}{2n}} \bigg) \bigg]^{\frac{2n}{2n+1}}
\end{align*}
where in the last equality we have used the fact that $f_n$-term is independent of $x_n$. By applying H\"{o}lder inequality with the conjugate pair of exponents $(2n, \frac{2n}{2n-1})$, 
we obtain 
\begin{align*}
    J \leq \frac{1}{q} \sum_{y \in \mathbb{F}_q} J_1(y) J_2(y)
\end{align*}
where 
\begin{align*}
    J_1(y)=& \bigg(\frac{1}{q^{2n-1}}\sum_{(\widehat{x}_{n, 2n},s) \in \mathbb{F}^{2n-1}_q } f_n(\widehat{x}_n,s)^{2n+1} \bigg)^{\frac{1}{2n+1}},\\
    J_2(y)=& \bigg(\frac{1}{q^{2n-1}}\sum_{(\widehat{x}_{n, 2n},s) \in \mathbb{F}_q^{2n-1} } \bigg(\frac{1}{q} \sum_{x_n \in \mathbb{F}_q} \prod_{\substack{j\neq n,2n}} f_j\big(\pi_j(x',y, s-\frac{1}{2} x_n y)^{\frac{2n+1}{2n}} \bigg)^{\frac{2n}{2n-1}} \bigg)^{\frac{2n-1}{2n+1}}
\end{align*}
Applying H\"{o}lder inequality  in the outer $y$--sum with the conjugate pair of  exponents $\left( 2n+1,\frac{2n+1}{2n} \right)$ yields
\begin{align*}
    J \lesssim &  \Big(\frac{1}{q}\sum_{y} J_1(y)^{2n+1}\Big)^{\frac{1}{2n+1}} \cdot
\Bigg(\frac{1}{q}\sum_{y\in \mathbb{F}_q} J_2(y)^{\frac{2n+1}{2n}}\Bigg)^{\frac{2n}{2n+1}} \\
=&  \bigg(\frac{1}{q^{2n}}\sum_{(\widehat{x}_{n},s) \in \mathbb{F}^{2n}_q } f_n(\widehat{x}_n,s)^{2n+1} \bigg)^{\frac{1}{2n+1}}\cdot J_{\Pi}\\
=&  \|f_n\|_{L^{2n+1}(\mathbb{F}_q^{2n}, dx)} \cdot J_{\Pi}.
\end{align*}
Therefore, it remains to bound
\[
J_\Pi := \Bigg(\frac{1}{q}\sum_{y\in \mathbb{F}_q} J_2(y)^{\frac{2n+1}{2n}}\Bigg)^{\frac{2n}{2n+1}}
\]
by the product of the remaining norms appearing on the right-hand side of \eqref{main-inequality}.

{\bf Step $3^{\circ}$ (Apply the induction hypothesis).}

To pass from dimension \(n\) to \(n-1\), we isolate the last pair of coordinates \((x_n,x_{2n})\) and treat it as a parameter. For each fixed choice of \((x_n,x_{2n})\in\F_q^2\), the remaining variables (together with the \(t\)-variable after the changes made in Steps \(1^\circ\)–\(2^\circ\)) carry the same Heisenberg structure as \(\mathbb H^{\,n-1}(\F_q)\). With this preparation, the bracketed expression in \eqref{inequality-J-pi} can be viewed as an \((n-1)\)--dimensional Loomis--Whitney type form in the remaining variables, applied to auxiliary factors obtained from the original \(f_j\)'s by freezing \((x_n,x_{2n})\). This allows us to invoke the induction hypothesis and obtain \eqref{inequality-J-pi-2}. The specific power \(\frac{2n+1}{2n-1}\) is chosen so that the norms produced by the \((n-1)\)--dimensional estimate translate, after averaging over \((x_n,x_{2n})\), into the desired \(L^{2n+1}\) and \(L^{\frac{2n+1}{2}}\) norms of the original functions.

By applying Minkowski's inequality, we obtain
\begin{align}\label{inequality-J-pi}
\nonumber
J_{\Pi}=& \Bigg(\frac{1}{q}\sum_{y\in \mathbb{F}_q} J_2(y)^{\frac{2n+1}{2n}}\Bigg)^{\frac{2n}{2n+1}}\\
\lesssim &\left(\frac{1}{q^2}   \sum_{x_{n}, y \in \mathbb{F}_q} \left[ \frac{1}{q^{2n-1}} \sum_{(\widehat{x}_{n, 2n},  t) \in \mathbb{F}_q^{2n-1}}  
\prod_{\substack{j\neq n,2n}}
f_{j}(\pi_{j}(x',y, t))^{\frac{2n+1}{2n-1}}    
 \right]^{\frac{2n-1}{2n}}\right)^{\frac{2n}{2n+1}}.
\end{align}
Observing that the expression inside the square brackets of \eqref{inequality-J-pi} is a product of $2n-2$ functions, we aim to transform it into a form suitable for applying the induction hypothesis to establish an upper bound for $J_{\Pi}$. 
We recall that 
\begin{equation}\label{definition-f-j}
f_{j}(\pi_{j}(x',y,\tau))=
\begin{cases}
f_{j}(\check{x}_{j}, \tau+\frac{1}{2}x_{j}x_{j+n}) , ~~~~~~~~~~~~~ \text{ if } 1 \leq j \leq n-1,\\
f_{j}(\check{x}_{n+j}, \tau-\frac{1}{2}x_{j-n}x_{j}) , ~~~\text{ if } n+1 \leq j \leq 2n-1.
\end{cases}
\end{equation}
We temporarily represent points in $\mathbb{H}^{n-1}(\mathbb F_q)$ using coordinates $(u, t)=(u_1, \ldots, u_{2n-2}, t)$ where $u \in \mathbb{F}_{q}^{2n-2}$. 
 Fixing $x_{n}, y\in \mathbb{F}_{q}$ we define the functions $g_{j}$, $j \in \{1, \ldots, 2n-2\}$, on $\mathbb{F}_{q}^{2n-2}$ by 
\begin{equation*}
    g_j(\widehat{u}_j,t) =
    \begin{cases}
            f_j(\tilde{u}_j)^{\tfrac{2n+1}{2n-1}}, ~~~~~~~~~~ 1 \leq j \leq n-1,\\
            f_{j+1}(\tilde{u}_{j})^{\tfrac{2n+1}{2n-1}}, ~~~~~n \leq j \leq 2n-2,
    \end{cases}
\end{equation*}
where \( \tilde{u}_j \) is obtained from 
\[
(u_1, u_2, \ldots, u_{n-1}, x_n, u_{n}, \ldots, u_{2n-2}, y, t)
\]
by removing \( u_j \). More precisely 
\begin{equation}\label{definition-g_n-2n-j-1}
g_{j}(\widehat{u}_{j},t)=
\begin{cases}
f_{1}(u_2, \ldots, u_{n-1}, x_{n}, u_{n}, \ldots, u_{2n-2}, y, t)^{\frac{2n+1}{2n-1}} , ~~~~~~~~~~~~~~~~~~~~~~~~~~~~~~~~~~j=1,\\
f_{j}(u_1, \ldots, u_{j-1}, u_{j+1}, \ldots, u_{n-1}, x_{n}, u_{n}, \ldots, u_{2n-2}, y, t)^{\frac{2n+1}{2n-1}} , ~~~2 \leq j \leq n-2,\\
f_{n-1}(u_1, \ldots, u_{n-2}, x_{n}, u_{n}, \ldots, u_{2n-2}, y, t)^{\frac{2n+1}{2n-1}} , ~~~~~~~~~~~~~~~~~~~~~~~~~ j=n-1,
\end{cases}
\end{equation}
and 
\begin{equation}\label{definition-g_n-2n-j-2}
g_{j}(\widehat{u}_{j},t)=
\begin{cases}
f_{n+1}(u_1, \ldots, u_{n-1}, x_{n}, u_{n+1}, \ldots, u_{2n-2}, y, t)^{\frac{2n+1}{2n-1}} , ~~~~~~~~~~~~~~~~~~~~~~~~~ j=n,\\
f_{j}(u_1, \ldots,  u_{n-1}, x_{n}, u_{n}, \ldots, u_{j-1}, u_{j+1}, \ldots, u_{2n-2}, y, t)^{\frac{2n+1}{2n-1}}, ~n+1 \leq j \leq 2n-3,\\
f_{2n-1}(u_1, \ldots, u_{n-2}, x_{n}, u_{n}, \ldots, u_{2n-3}, y, t)^{\frac{2n+1}{2n-1}} , ~~~~~~~~~~~~~~~~~~~~~~~~~ j=2n-2.
\end{cases}
\end{equation}
With this notation in place, and recalling \eqref{definition-f-j}, we can restate \eqref{inequality-J-pi} equivalently as follows
\begin{align*}
\nonumber
J_{\Pi} \lesssim &\left(\frac{1}{q^2} \sum_{x_{n},y\in \mathbb{F}_q} 
\left[\frac{1}{q^{2n-1}} \sum_{(u, t)\in \mathbb{F}_q^{2n-1}} \prod_{\substack{1 \leq j \leq 2n-2}} g_{j}(\pi_{j}(u, t))    \right]^{\frac{2n-1}{2n}}
\right)^{\frac{2n}{2n+1}} 
\end{align*}
where we recall that $\pi_{j}$ is the projection from $\mathbb{H}^{n-1}(\mathbb F_q)$ to the vertical plane $W_j$ (identified with $ \mathbb F_{q}^{2n-2}$), defined in \eqref{definition-pi-j}. 
By the induction hypothesis we have
\begin{align}\label{inequality-J-pi-2}
\nonumber
\frac{1}{q^{2n-1}} &\sum_{(u,t) \in \mathbb{F}_q^{2n-1}} \prod_{\substack{1 \leq j \leq 2n-2}} g_j(\pi_j(u,t)) \\
\;\lesssim\; 
&
\|g_1\|_{L^{\frac{2n-1}{2}}(\mathbb{F}_q^{2n-2}, dx)}\,
\|g_{n}\|_{L^{\frac{2n-1}{2}}(\mathbb{F}_q^{2n-2}, dx)}\,
\prod_{\substack{1 \leq j \leq 2n-2,\\ j\ne 1,n}} \|g_j\|_{L^{2n-1}(\mathbb{F}_q^{2n-2}, dx)}. 
\end{align}
Next we apply the multilinear H\"{o}lder's inequality with exponents 
$\ell_1=\ell_n=n$ and $\ell_j=2n$, $\forall\, j \ne 1, n$.  Note that 
$
\sum\limits_{j=1}^{2n-2} \frac{1}{\ell_{j}}=1
$.
Hence, from \eqref{inequality-J-pi-2}, we have 
\begin{align*}
J_{\Pi} \lesssim &\bigg[\bigg( \frac{1}{q^2}  \sum_{x_{n}, y\in \mathbb{F}_q}  \| g_{1}  \|_{L^{\frac{2n-1}{2}}(\mathbb{F}_q^{2n-2}, dx)}^{\frac{2n-1}{2}} \bigg)^{\frac{1}{n}}  \cdot 
\bigg( \frac{1}{q^2} \sum_{x_{n}, y\in \mathbb{F}_q}   \| g_{n}  \|_{L^{\frac{2n-1}{2}}(\mathbb{F}_q^{2n-2}, dx)}^{\frac{2n-1}{2}} \bigg)^{\frac{1}{n}} \\
& \cdot 
\prod_{\substack{1 \leq j \leq 2n-2,\\ j\ne 1,n}} 
\bigg( \frac{1}{q^2}  \sum_{x_{n}, y\in \mathbb{F}_q}  \| g_{j}  \|_{L^{2n-1}(\mathbb{F}_q^{2n-2}, dx)}^{2n-1}\bigg)^{\frac{1}{2n}} \bigg]^{\frac{2n}{2n+1}}\\
=&  \|f_{1} \|_{L^{\frac{2n+1}{2}}(\mathbb{F}_q^{2n}, dx)}  \|f_{n+1} \|_{L^{\frac{2n+1}{2}}(\mathbb{F}_q^{2n}, dx)} 
 \prod_{j=2}^{n-1}\left(  \|f_{j} \|_{L^{2n+1}(\mathbb{F}_q^{2n}, dx)}  \|f_{n+j} \|_{L^{2n+1}(\mathbb{F}_q^{2n}, dx)} \right).
\end{align*} 
Here the last equality follows from the definition of $g_j$ for $j \in \{1, \ldots, 2n-2\}$, given in \eqref{definition-g_n-2n-j-1} and \eqref{definition-g_n-2n-j-2}.

\end{proof}

For a positive $n$, we define a multilinear map \[T: \mathscr{F}(\mathbb F_q^{2n})\times \cdots \times \mathscr{F}(\mathbb F_q^{2n}) \rightarrow \mathscr{F}(\mathbb F_q^{2n})\] as follows: For an element $(z,\tau)=(z_1,\ldots,z_{2n-1},\tau)\in \mathbb F_q^{2n}$
\begin{align*}
T(g_1, \ldots, g_{2n-1})(z, \tau):=\frac{1}{q}
 \sum_{s \in \mathbb F_q} 
\left[g_n \big(z, \tau+ z_{n}s \big) \cdot \prod_{\substack{ j=1, j \ne n}}^{2n-1} g_{j} \left(\pi_j \big(z, s, \tau+\frac{1}{2}z_n s \big) \right)\right].
\end{align*}
We use 
this operator $T$ to find an  upper bound of  $I$,
given in \eqref{definition-I}.  More precisely,
taking a change of variable $\tau =t+\frac{1}{2}x_n x_{2n}$, and applying H\"{o}lder's inequality with the conjugate pair of exponents
\begin{align}\label{dual-exponent}
(\ell, \ell^\prime)=\bigg(\frac{n(2n+1)}{n+1},\frac{n(2n+1)}{2n^2-1} \bigg),
\end{align}
we obtain 
\begin{align}\label{estimate-1}
\nonumber
I=&\frac{1}{q^{2n+1}}
\sum_{(x,t)\in 
\mathbb H^n(\mathbb F_q)}
\prod_{j=1}^{2n} f_j(\pi_j(x,t)) \\
\nonumber
=&\frac{1}{q^{2n}}\sum_{(x',t) \in \mathbb{F}_q^{2n}} \sum_{x_{2n}  \in \mathbb{F}_q}
f_{2n}(x',t-\frac{1}{2}x_n x_{2n})\,
\Bigg(\frac{1}{q} f_n(\widehat{x}_{n},t + \frac{1}{2} x_n x_{2n})
\prod_{\substack{j\neq n,2n}}
f_j\big(\pi_j(x,t)\big)\Bigg)\\ \nonumber
=&\frac{1}{q^{2n}}\sum_{(x', \tau) \in \mathbb{F}_q^{2n}} 
f_{2n}(x', \tau)\,
\Bigg(\frac{1}{q} \sum_{x_{2n}  \in \mathbb{F}_q} f_n(\widehat{x}_{n}, \tau - \frac{1}{2} x_n x_{2n})
\prod_{\substack{j\neq n,2n}}
f_j\big(\pi_j(x', x_{2n},\tau+\frac{1}{2}x_n x_{2n})\big)\Bigg)\\
\leq & \| T(f_1, \ldots, f_{2n-1})\|_{L^{\ell'}(\mathbb{F}^{2n}_q, dx) }  \cdot 
\| f_{2n}\|_{ L^{\ell} (\mathbb{F}^{2n}_q, dx) }.
\end{align}

We are now ready to prove Theorem \ref{high-dimension-2}.
\begin{proof}[Proof of Theorem \ref{high-dimension-2}]
In view of \eqref{estimate-1}, the proof of Theorem \ref{high-dimension-2} is reduced to proving 
\begin{align}\label{estimate-2}
 \| T(f_1, \ldots, f_{2n-1})\|_{L^{\ell^\prime}(\mathbb{F}^{2n}_q, dx) } 
 \lesssim
 \prod_{j=1}^{2n-1} \| f_{j}\|_{L^{\ell_j}(\mathbb{F}^{2n}_q, dx) } 
\end{align} for all 
nonnegative
$f_1, \ldots, f_{2n-1}\in \mathscr{\mathbb F_q^{2n}}$,
where $\ell^\prime$ is defined in \eqref{dual-exponent} and $\ell_1=\ldots=\ell_{2n-1}=\frac{n(2n+1)}{n+1}$.
To prove the inequality \eqref{estimate-2}, 
first, from \eqref{main-inequality} we obtain
\begin{align}\label{estimate-3}
 \| T(f_1, \ldots, f_{2n-1})\|_{L^{\ell_k}(\mathbb{F}^{2n}_q, dx) } \lesssim \prod_{j=1}^{2n-1} \| f_{j}\|_{L^{\ell_{j,k}}(\mathbb{F}^{2n}_q, dx)}, ~~~k=1, \ldots, n
\end{align} 
where $\ell_{k}$, $\ell_{j,k}$ are given by 
\begin{equation}\label{definition-q_k}
\ell_k=
\begin{cases}
\frac{2n+1}{2n} , ~~~~~~~~~~k=1, \ldots, n-1,\\
\frac{2n+1}{2n-1} , ~~~~~~~~~~k=n
\end{cases}
\end{equation}
and 
\begin{equation}\label{definition-p-j_k}
\ell_{j, k}=
\begin{cases}
2n+1 , ~~~~~~~~~~~~~~k \notin \{j, j+n, j-n\},\\
\frac{2n+1}{2} , ~~~~~~~~~~~~~~~k \in \{j, j+n, j-n\},\\
\end{cases}
\end{equation}
$j =\{1, \ldots, 2n-1\}$, $k=1, \ldots, n$.

Note that the exponents in \eqref{definition-q_k} and \eqref{definition-p-j_k} satisfy
\begin{align*}
    \frac{1}{\ell^\prime}=\frac{2n^2-1}{n(2n+1)}=\frac{1}{n} \bigg[ \frac{2n}{2n+1} \cdot (n-1)+\frac{2n-1}{2n} \bigg]=\sum\limits_{k=1}^{n} \frac{1}{n} \frac{1}{\ell_k}
\end{align*}
and
\begin{align*}
    \frac{1}{\ell_j}=\sum\limits_{k=1}^{n} \frac{1}{n} \frac{1}{\ell_{j,k}}, ~~~~~~~~j=1, \ldots, 2n-1.
\end{align*}
To prove \eqref{estimate-2}, we apply the discrete version of the multilinear interpolation theorem, 
a generalization to \cite[Proposition 6.1]{kohpham}. 
Since $(\mathbb{F}_q^{2n}, dx)$ is a finite measure space equipped with the normalized counting measure, 
the same interpolation argument carries over verbatim. 
We then perform $(n-1)$ iterations with $\theta = \tfrac{1}{k}$, starting from $k=2$ and ending at $k=n$, 
using the identity
\[
\frac{1}{k}[a_1+\cdots+a_k] 
= \left(1-\frac{1}{k}\right)\left(\frac{1}{k-1}[a_1+\cdots+a_{k-1}] + \frac{1}{k}a_k \right).
\]
More precisely, we first apply multilinear interpolation theorem with $\theta=\frac{1}{2}$ to the two operator bounds given by \eqref{estimate-3} for $k=1$ and $k=2$. Next, we apply the multilinear interpolation theorem 
with $\theta=\frac{1}{3}$ to interpolate between the newly obtained bound and the operator bound stated in \eqref{estimate-3} for $k=3$. We continue this process until, in the final step, we apply the theorem with $\theta=\frac{1}{n}$ to interpolate between the previously obtained bound and the bound for $k=n$. From this, we obtain \eqref{estimate-2} for all nonnegative 
$f_1, \ldots, f_{2n} \in \mathscr{F}(\mathbb{F}^{2n}_q)$.

\end{proof}
\section{Proof of Corollary \ref{Loomis-Whitney-inequality-F_p^q}}

\begin{proof}
For  subsets \(K_1, \ldots, K_{2n} \subset \mathbb{F}_q^{2n}\), denote \[K:=\bigcap\limits_{j=1}^{2n} \pi_j^{-1}(K_j),\]
where each  $K_j$ is viewed as a subset of $\mathbb H^n(\mathbb F_q)$ through the identification  $\mathbb F_q^{2n}=W_j \subset \mathbb H^n(\mathbb F_q)$ for $j\in \{1, \ldots 2n\}$.
Then for $(x,t)\in \mathbb H^n(\mathbb F_q),$ we can write
\[
\mathbf{1}_{ K}(x,t) 
= \prod_{j=1}^{2n} \mathbf{1}_{K_j}(\pi_j(x,t)).
\]
Then, to prove Corollary \ref{Loomis-Whitney-inequality-F_p^q}, it is sufficient to prove the inequality
\begin{align}\label{inequality-theorem1.9}
\frac{1}{q^{2n+1}} \sum_{(x,t) \in \mathbb{H}^n (\mathbb{F}_q)}
\prod_{j=1}^{2n} \mathbf{1}_{K_j}(\pi_j(x,t))
\;\;\lesssim\;\; 
\prod_{j=1}^{2n} \|\mathbf{1}_{K_j}\|_{L^{\frac{n(2n+1)}{n+1}}(\mathbb{F}_q^{2n}, dx)}.
\end{align}
Note that
\[
\|\mathbf{1}_{K_j}\|_{L^p} = \bigg( \frac{|K_j|}{q^{2n}}\bigg)^{1/p}, 
\qquad j=1,\dots,2n.
\]
Finally, applying Theorem \ref{high-dimension-2} with 
\(f_j = \mathbf{1}_{K_j}\) for $j\in \{1, \ldots, 2n\}$, multiply both sides by $q^{2n+1}$ and simplify the power of $q$ to obtain \eqref{inequality-theorem1.9}, as desired.
\end{proof}


\section{Proof of Theorem \ref{Theorem-L-W-ineq:q-q^1/2}}
Let $ \mathscr P$ be a set of points in $\mathbb{F}_q^2$ and $\mathscr L$ be a set of lines in $\mathbb{F}_q^2$. The number of incidences between $ \mathscr P$ and $\mathscr L$, denoted by $I( \mathscr P, \mathscr L)$, is defined by 
\[I(\mathscr P, \mathscr L)=|\{(p, \ell)\in \mathscr P\times \mathscr L\colon p\in \ell\}|.\]
To prove Theorem \ref{Theorem-L-W-ineq:q-q^1/2}, we invoke
 the following incidence bound due to Vinh \cite{Vinh}.


\begin{theorem}\label{big}
    Let $ \mathscr P$ be a set of points in $\mathbb{F}_q^2$ and $\mathscr L$  a set of lines in $\mathbb{F}_q^2$. The number of incidences between $ \mathscr P$ and $\mathscr L$ satisfies
    \[I(\mathscr P, \mathscr L) \le  \frac{|\mathscr P| \, |\mathscr L|}{q}+2q^{\frac{1}{2}}\sqrt{|\mathscr P| \, | \mathscr L|}.\]
\end{theorem}

\begin{proof}[Proof of Theorem \ref{Theorem-L-W-ineq:q-q^1/2}]
As in Section \ref{section-1}, we define 
a bijective
map $\phi\colon \mathbb{F}_q^3\to \mathbb{F}_q^3$ by $\phi(x, y, t)=(x, y, t+\frac{xy}{2})$, then 
\begin{align*}
    &\frac{1}{q^3}\sum_{(x_1, x_2, t)\in \mathbb{H}_q^1}
f_1(\pi_1(x_1, x_2, t))f_2(\pi_2(x_1, x_2, t))=\frac{1}{q^3}\sum_{x_1, x_2, t\in \mathbb{F}_q}f_1(\pi_1(\phi(x_1, x_2, t)))f_2(\pi_2(\phi(x_1, x_2, t)))\\
    &=\frac{1}{q^3}\sum_{x_1, x_2, t\in \mathbb{F}_q}f_1(x_2, t+x_1x_2)f_2(x_1, t)\\
    &=\frac{1}{q^3}\sum_{ (a_1, a_2), (b_1, b_2) \in \mathbb{F}^2_{q}, a_1 b_1+b_2=a_2}f_1(a_1, a_2)f_2(b_1, b_2).
\end{align*}
Let $f_1$ and $f_2$ be the indicator functions of $\pi_1(K)$ and $\pi_2(K)$, respectively.
Each element $(x_1,x_2,t)\in K$ determines a point
$(a_1,a_2) = \pi_1(x_1,x_2,t) \in P$
and a line $ \ell = \pi_2(x_1,x_2,t) \in L$, 
which satisfy the relation $a_1 b_1 + b_2 = a_2$.
Hence each element of $K$ gives rise to an incidence between $P$ and $L$.

Moreover, when $n=1$, the map
\[
(x_1,x_2,t) \longmapsto \bigl(\pi_1(x_1,x_2,t),\,\pi_2(x_1,x_2,t)\bigr)
\]
is injective, so distinct elements of $K$ yield distinct incidences.
Therefore,
\[
|K| \le I(P,L).
\]
Applying Theorem \ref{big}, we obtain
\[
\frac{|K|}{q^3}
\le
\frac{|\pi_1(K)|\,|\pi_2(K)|}{q^4}
+
2 q^{-5/2}\sqrt{|\pi_1(K)|\,|\pi_2(K)|}.
\]
This yields the desired conclusion.
\end{proof}

\section{Connections to other topics}\label{section7}
In this section, we address how our results are connected to some questions in the literature.

\subsection{Bhowmik, Iosevich, Koh, and Pham's boundedness problem}
Let $\mathcal{K}\colon \mathbb{F}_q^d\times \mathbb{F}_q^d\to \mathbb{R}$ be any function (which, below,  will play the role of the kernel of an integral operator for the counting measure), and  $\mathcal{G}=(V, E)$ a connected ordered graph on $m$ vertices. In a recent paper \cite{BIKP2023}, Bhowmik, Iosevich, Koh, and Pham introduced  the boundedness problem for a multilinear operator associated with $\mathcal{K}$ and $\mathcal{G}$. To be precise, define
\[ \Lambda=\Lambda_\mathcal{G}^\mathcal{K}: \mathscr{F}(\mathbb F_q^d) \times \cdots \times \mathscr{F}(\mathbb F_q^d) \rightarrow \mathbb R\] by
\begin{equation*}\label{op:1}
(f_1, \ldots, f_m) \mapsto \frac{1}{\mathcal{N}(\mathcal{G})}\sum_{x^1, \dots, x^m \in {\Bbb F}_q^d} \ \prod_{(i,j) \in {E}(\mathcal{G})} \mathcal{K}(x^i,x^j) \prod_{l=1}^m f_l(x^l),
\end{equation*}
where $\mathcal{N}(\mathcal{G})$ is the number of distinct embeddings of $\mathcal{G}$ in $\mathbb{F}_q^d$. 
The boundedness problem for  operator $\Lambda_\mathcal{G}^\mathcal{K}$ asks if one can determine  numbers $1\le  u_1, \ldots, u_m \le \infty$ for which there is a (smallest) constant $ \Lambda_\mathcal{G}^\mathcal{K} (u_1, \dots, u_m)> 0$ such that the inequality 
 \begin{equation} \label{generalproblem} 
 \frac{1}{\mathcal{N}(\mathcal{G})}\sum_{x^1, \dots, x^m \in {\Bbb F}_q^d} \ \prod_{(i,j) \in {E}(\mathcal{G})} \mathcal{K}(x^i,x^j) \prod_{l=1}^m f_l(x^l)\leq  \Lambda_\mathcal{G}^\mathcal{K}  (u_1, \dots, u_m) \prod_{i=1}^m {\|f_i\|}_{L^{u_i}(\mathbb{F}_q^d, dx)} 
 \end{equation}
holds for all nonnegative
$ f_1, \ldots, f_m \in \mathscr{F}(\mathbb{F}_q^d)$. Using the discrete Fourier analysis, they settled the case where  $\mathcal{G}$ is a connected graph with at most $4$ vertices and $\mathcal{K}(x, y)$ is defined by $\mathcal{K}(x, y)=1$ if $\|x-y\|=(x_1-y_1)^2+(x_2-y_2)^2=1$, and $\mathcal{K}(x, y)=0$ otherwise.

The operator $\Lambda_\mathcal{G}^\mathcal{K}$ is a generalization of several geometric counting functions: 
\begin{enumerate}
    \item If $\mathcal{G}=K_2$, and $f_1=f_2$ are the indicator functions of  a set $A$ in $\mathbb{F}_q^d$, then the operator $\Lambda_\mathcal{G}^\mathcal{K}$   
    counts the number of pairs of points in $A\times A$ of distance $1$. 
    \item If $\mathcal{G}=C_4$, and $f_1=f_2=f_3=f_4$ are the indicator functions of  a set $A$ in $\mathbb{F}_q^d$, then the operator $\Lambda_\mathcal{G}^\mathcal{K}$  
    counts the number of rhombi of side-length $1$ in $A$. 
    \item If $\mathcal{G}$ is a tree and $f_1=\cdots=f_m$ are the indicator functions of  a set $A$ in $\mathbb{F}_q^d$, then the operator $\Lambda_\mathcal{G}^{\mathcal{K}}$ 
    counts the number of copies of the tree $\mathcal{G}$ in $A$. 
\end{enumerate}

The inequality \eqref{LW11} for $n=1$ is essentially the same as the inequality \eqref{generalproblem} for $m=2$, with  suitable  $\mathcal{G}$ and $\mathcal{K}$ chosen. 
To see this, define $\phi\colon \mathbb{F}_q^3\to \mathbb{F}_q^3$ by $\phi(x, y, t)=(x, y, t+\frac{xy}{2})$.
Then, since $\phi$ is a bijective map, the left-hand side of \eqref{LW11} can be written 
\begin{align}\label{connection1}
\nonumber
    &\frac{1}{q^3}\sum_{(x_1, x_2, t)\in \mathbb{H}_q^1}
f_1(\pi_1(x_1, x_2, t))f_2(\pi_2(x_1, x_2, t))=\frac{1}{q^3}\sum_{x_1, x_2, t\in \mathbb{F}_q}f_1(\pi_1(\phi(x_1, x_2, t)))f_2(\pi_2(\phi(x_1, x_2, t)))\\
    &=\frac{1}{q^3}\sum_{x_1, x_2, t\in \mathbb{F}_q}f_1(x_2, t+x_1x_2)f_2(x_1, t).
\end{align}
Now let $\mathcal{G}$ be the complete graph of two vertices and define $\mathcal{K}$ by
\[\mathcal{K}((a_1, a_2), (b_1, b_2))=\begin{cases}
    1~&\mbox{if}~a_1b_1+b_2-a_2=0,\\
    0~&\mbox{otherwise} .
\end{cases}
\]

Then the left-hand side $\Lambda_\mathcal{G}^\mathcal{K}(f_1, f_2)$ of the inequality \eqref{generalproblem} is equal to 
\[\frac{1}{q^3}\sum_{ (a_1, a_2), (b_1, b_2) \in \mathbb{F}^2_{q}, a_1 b_1+b_2=a_2}f_1(a_1, a_2)f_2(b_1, b_2),\]
which coincides with the value \eqref{connection1}. Thus, in this case, both inequalities \eqref{LW11}   and \eqref{generalproblem} coincide with each other.
This connection suggests a further study of the operator $\Lambda_\mathcal{G}^\mathcal{K}$ for various  $\mathcal{K}$ and $\mathcal{G}$.

\subsection{The orthogonal projection problem} 
For a subset $K\subset\F_q^{2n+1}$ and an integer $r\geq 1$, let
\begin{align*}
    \mathscr{E}_r=\mathscr E_r(K) \; := & \; \big\{\, W\le\F_q^{2n+1}:\ \dim W=2n,\;\; K \text{ can be covered by at most } r \\
    &~~~~ \text{ (additive) translates of } W^\perp \,\big\},
\end{align*}
where a translate of $W^\perp$ means a subset of $\mathbb F_q^{2n+1}$ of the form $(x,t)+W^\perp$ for some $(u,\tau)\in \mathbb F_q^{2n+1}$
Chen~\cite{Chen} proved that
\begin{equation}\label{eq:Chen}
|\mathscr{E}_r|
\;\le\;
\begin{cases}
q^{2n-1}\,r, & \text{if } r\le \dfrac{|K|}{2},\\[2mm]
\dfrac{r\,q^{4n}}{(q^{2n}-r)\,|K|}, & \text{if } 0<r<q^{2n}.
\end{cases}
\end{equation}
Informally, this states that for most subspaces $W\subset \mathbb F_q^{2n+1}$, the \emph{covering number} of $K$ by translates of $W^\perp$ is not very small. 

Chen's result can be interpreted from the view of the (orthogonal) projection problem. To see this, for each $j=1, \ldots, 2n$, let $P_j: \mathbb F_q^{2n+1}\mapsto W_j$ be the orthogonal projection defined by
$P_j(x,t):=(\check{x}_j,t),$ where $W_j$ is the coordinate hyperplane given in~\eqref{coordinate-hyperplanes}.   Note that  $W_j^\perp=\mathrm{Ker}\, P_j$ and the translates of $W_j^\perp$ are exactly cosets of the subspace $W_j^\perp$ in $\mathbb F_q^{2n+1},$ which are in turn the fibres of $P_j$. Note that  each translate of $W_j^{\perp}$ can be  written in the form $(u,\tau)+W_j^{\perp}$ for uniquely determined $(u,\tau)\in W_j.$ Thus, we observe that for a subset $K$ of $\mathbb F_q^{2n+1}$,  $K\cap ((u,\tau)+W_j^{\perp})\ne \emptyset$ for $(u,\tau)\in W_j$ if and only if $(u,\tau)\in P_j(K)$. 
Thus, one can write 
\begin{equation}\label{covering-0}
K \;\subset\; \bigcup_{(u,\tau) \in P_j(K)}  \left( (u,\tau)+ W_j^{\perp}\right).
\end{equation}
This implies that the minimal number of translates of $W_j^\perp$ that cover a set $K \subset \mathbb{F}_q^{2n+1}$ coincides with $|P_j(K)|.$

Now let $\mathscr{H}=\{W_1,\dots,W_{2n}\}$ be the set of coordinate hyperplanes in $\mathbb F_q^{2n+1}$. If $|\mathscr H| = 2n > |\mathscr E_r|$, then, by the definition of $\mathscr E_r$, there exists at least one coordinate  hyperplane $W_j\in \mathscr H$ that does not belong to $\mathscr E_r$, and for this $j$ the covering number of $K$ by translates of $W_j^\perp$ is strictly larger than $r$, i.e., $P_j(K)$ contains more than $r$ points.

 Combining this observation with Chen's estimate \eqref{eq:Chen} and the fact that $|\mathscr{H}|=2n$, one finds a constant $c_n>0$ depending only on $n$ such that for every nonempty set $K\subset\mathbb{F}_q^{2n+1}$ there exists an index $j_0\in\{1,\dots,2n\}$ with
\begin{equation}\label{eq-x11}
|P_{j_0}(K)| \;\ge\; c_n\,\frac{|K|}{q^{2n}}.
\end{equation}
Thus, among the given $2n$  projections $P_{j}: \mathbb F_q^{2n+1}\mapsto W_j$ for $j=1,\ldots, 2n,$ there exists at least one projection $P_{j_0}$ under which the image of $K$ must have the size comparable to $|K|/q^{2n}$.

This discussion can be phrased in a more general framework. More precisely, let $(H, \cdot)$ be a (finite) set equipped with an algebraic structure with the identity and let $\{\Phi_i: H\mapsto V_i\, |\,i\in I\}$ be a family of algebraic projections, where $V_i$ are subobjects of $H$ with the same cardinality. Let $Z_i:=Ker \, \Phi_i$. Then the situation is analogous to the case of the orthogonal projections $P_j$: Each coset of $Z_j$ can be expressed in  the form $v\cdot Z_j$ for uniquely determined $v\in V_j,$ and a coset $v\cdot Z_j$ for $v\in V_j$  is equal to the fiber $\Phi_i^{-1}(v).$
Also we have that   for a subset $K$ of $H$,  $K\cap v\cdot Z_i\ne \emptyset$ for $v\in V_i$ if and only if $v\in \Phi_i(K)$. 
Thus, one can write
\begin{equation}\label{covering}
K \;\subset\; \bigcup_{v \in\Phi_i(K)} v\cdot Z.
\end{equation}
In particular, the minimal number of  cosets of $Z_i$ in $H$ that cover $K$ is equal to $|\Phi_i(K)|.$
In the light of \eqref{eq-x11}, one seek an index $i_0\in I$ for which the image $\Phi_{i_0}(K) \subset V_{i_0}$ is as large as possible. In practice, depending on which features of the original set one wishes to preserve, the choice of an appropriate family of projection maps is crucial. Our case concerns $H=\mathbb H^n(\mathbb F_q)$ and $\Phi_i=\pi_i$  for $i\in I=\{1,\ldots, 2n\}.$

In the following, we compare Chen's theorem and Corollary~\ref{Loomis-Whitney-inequality-F_p^q} from this point of view. For each $j=1, \ldots, 2n$, define $T_j:\mathbb{H}^n(\mathbb{F}_q)\to\F_q^{2n+1}$ by
\begin{equation}\label{eq:Tj}
T_j(x,t)=
\begin{cases}
(x,\ t+\tfrac12\,x_j x_{n+j}), & 1\le j\le n,\\[2pt]
(x,\ t-\tfrac12\,x_{j-n}\,x_j), & n+1\le j\le 2n.
\end{cases}
\end{equation}
This map is a bijective map with $P_j\circ T_j=\pi_i$ for all $j=1,\ldots, 2n$. We observe that for each $(u,\tau)\in W_j$, we have
\begin{equation}\label{inclusion}
T_j\big(\pi_j^{-1}(u,\tau)\big)=P_j^{-1}(u,\tau),
\end{equation}
i.e., $T_j$ maps each Heisenberg fibre $\pi^{-1}(u, \tau)$ bijectively onto the fibre $P_j^{-1}(u,\tau) $. Moreover, the restriction of $T_j$ to $W_j$ is the identity map.

  Let $K\subset \mathbb \mathbb H^n(\mathbb F_q)$ and fix $j=1, \ldots, 2n.$ Recall that $\mathrm{Ker}\, \pi_j=W_j^{\perp}$. Letting $H=\mathbb H^n(\mathbb F_q)$ and $\Phi_j=\pi_j$,  by \eqref{covering} we have
 \begin{equation}\label{eq-x33}
K \;\subset\; \bigcup_{(u,\tau)\in\pi_j(K)} (u,\tau)\cdot W_j^\perp.
\end{equation}
  Applying  $T_j$ on both sides of \eqref{eq-x33}, by the equality \eqref{inclusion} and the facts that $\pi_j^{-1}(u,\tau)=(u,\tau)\cdot W_j^{\perp}$ and $P_j^{-1}(u,\tau)=(u,\tau)+W_j^{-1}$, we obtain

\[
T_j(K) \;\subset\; \bigcup_{(u,\tau)\in\pi_j(K)} \bigl( (u,\tau)+W_j^\perp \bigr).
\]

Note that  each coset $(u,\tau)+W_j^{\perp}$ in the above union is necessary to cover $T_j(K).$ Thus, the quantity $|\pi_j(K)|$ coincides with the covering number of $T_j(K)$ by translates of $W_j^\perp$.
Now for each $j=1, \ldots, 2n$, define
\[
\phi_j \;:=\; T_j^{-1}\circ P_j\circ T_j : \mathbb{H}^n(\mathbb{F}_q) (\cong \mathbb{F}_q^{2n+1})\to \mathbb{H}^n(\mathbb{F}_q) (\cong \mathbb{F}_q^{2n+1}).
\]
These maps enjoy three basic features that we use repeatedly: they are idempotent retractions,
\[
\phi_j\circ\phi_j=\phi_j;
\]
they have the same fibres as $\pi_j$, in the sense that $\phi_j(x)=\phi_j(y)$ if and only if $\pi_j(x)=\pi_j(y)$; and, consequently, for every subset $K\subset \mathbb{H}^n(\mathbb{F}_q)$,
\[
|\phi_j(K)|=|\pi_j(K)|=\bigl|P_j\bigl(T_j(K)\bigr)\bigr|.
\]
Corollary~\ref{Loomis-Whitney-inequality-F_p^q} yields a constant $C_n>0$ such that
\[
\max_{1\le j\le 2n}|\pi_j(K)| \;\ge\; C_n\,|K|^{\frac{2n+1}{2(n+1)}} q^{-\frac{1}{2(n+1)}},
\]
and hence there exists an index $j^\prime\in\{1,\dots,2n\}$ for which
\begin{equation}\label{eq-x22}
|\phi_{j'}(K)| \;=\; |P_{j'}(T_{j'}(K))| \;\ge\; C_n\,|K|^{\frac{2n+1}{2(n+1)}} q^{-\frac{1}{2(n+1)}}.
\end{equation}
We now compare between (\ref{eq-x11}) and (\ref{eq-x22}). Using only the trivial constraint $|K| \le q^{2n+1}$, one checks that
\[
\frac{|K|^{\frac{2n+1}{2(n+1)}} q^{-\frac{1}{2(n+1)}}}{|K|/q^{2n}} \;\ge\; q^{2n-1}
\]
for all nonempty $K \subset \mathbb{H}^n(\mathbb{F}_q) (\cong \mathbb{F}_q^{2n+1})$.

Thus, with the same $\mathscr{H}$, the family $\{\pi_j\}_j$ of projections provides a strictly stronger lower bound for the maximal projection size in one of the vertical directions.

Another related question is a discrete analogue to the problem considered by Balogh, Fässler, Mattila, and Tyson~\cite{Balog} for the real Heisenberg group. To explain this, for
$K\subset \mathbb{H}^n(\mathbb{F}_q)$ and $1\leq m \leq n$, let $\mathscr{E}_r^{h,m}(K)$ denote the set of homogeneous subgroups $G\subset \mathbb{H}^n(\mathbb{F}_q)$ of size $q^m$  such that $K$ can be covered by at most $r$ left cosets of $G^\perp$. (See Section \ref{section2-1} for the notion \emph{homogeneous} and the notation $G^{\perp}$.) Then following Balogh, Fässler, Mattila, and Tyson, it is natural to seek upper bounds for $|\mathscr{E}_r^{h,m}(K)|$ in terms of $|K|$, $q$, $n$, and $m$. 
We plan to pursue this Heisenberg covering problem in a sequel paper.

\section{Acknowledgements}
Thang Pham and Dung The Tran would like to thank VIASM for the hospitality and for the excellent working conditions. Daewoong Cheong  sincerely thanks the Korea Institute for Advanced Study for financial support and hospitality and was partially supported by the research year program  of Chungbuk National University in 2025.

\end{document}